\documentclass[11pt,a4paper]{amsart}
\usepackage{amsmath}
\usepackage[utf8]{inputenc}
\usepackage[english]{babel}
\usepackage[T1]{fontenc}
\usepackage{url}
\usepackage{mathrsfs}
\usepackage{ifthen}

\usepackage{tikz}
\usetikzlibrary{knots}

\usepackage{ amssymb }
\usepackage{lscape}
\usepackage{multirow}
\usepackage{wrapfig} 

\usepackage{caption}

\newcommand \ifdraw {\iffalse}

\newcommand \X  {\mathfrak{X}}
\newcommand{\cros}{\operatorname{cr}}
\newcommand{\crosmnd}{\operatorname{cr}^{mnd}}
\newcommand{\crospot}{\operatorname{cr}^{pot}}
\newcommand{\estim}{\frac{208}{431}\sqrt[9]{\frac{8}{431}}\cdot\left(\sqrt[9]{53\frac{7}{8}}\right)^{\cros(K)}}

\newtheorem{thm}{Theorem}
\newtheorem{corollary}{Corollary}

\theoremstyle{definition}
\newtheorem{definition}{Definition}

\newtheorem{example}{Example}

\theoremstyle{remark}
\newtheorem{remark}{Remark}

\title[On the complexity of meander-like diagrams of knots]{On the complexity of meander-like diagrams of knots}
\author{Yury Belousov}
    \thanks{This work was supported by the Russian Science Foundation under grant no. 25-11-00251.}

\address{Yury Belousov\\
Saint Petersburg State University}
\email{bus99@yandex.ru}

\begin{document}
\maketitle

\begin{abstract}
   It is known that each knot has a semimeander diagram (i.\,e. a diagram composed of two smooth simple arcs), however the number of crossings in such a diagram can only be roughly estimated. In the present paper we provide a new estimate of the complexity of the semimeander diagrams. We prove that for each knot $K$ with more than 10 crossings, there exists a semimeander diagram with no more than $0.31 \cdot 1.558^{\cros(K)}$ crossings, where $\cros(K)$ is the crossing number of $K$. 
   As a corollary, we provide new estimates of the complexity of other meander-like types of knot diagrams, such as meander diagrams and potholders. We also describe an efficient algorithm for constructing a semimeander diagram from a given one.
\end{abstract} 

\maketitle
\section{Introduction}
The diagrams of knots which are composed of two smooth simple arcs (see examples on Fig.~\ref{fig:semimeander}) were studied under different names by many authors. Apparently, the fact that each knot has such a diagram was first proved by G.~Hots in 1960 in~\cite{H60} (for a detailed historical background on this subject see~\cite{BM20}). We call such diagrams \emph{semimeander diagrams}\footnote{This definition refers to semimeanders --- objects from combinatorics that are a configuration of a ray and a simple curve in the plane; as opposed to meanders, which are a configuration of a straight line and a simple curve in the plane (details on the combinatorics of meanders and semimeanders can be found, for example, in~\cite{DFGG97}).}.

\begin{figure}[h]
    \centering
    \includegraphics[scale = 0.19]{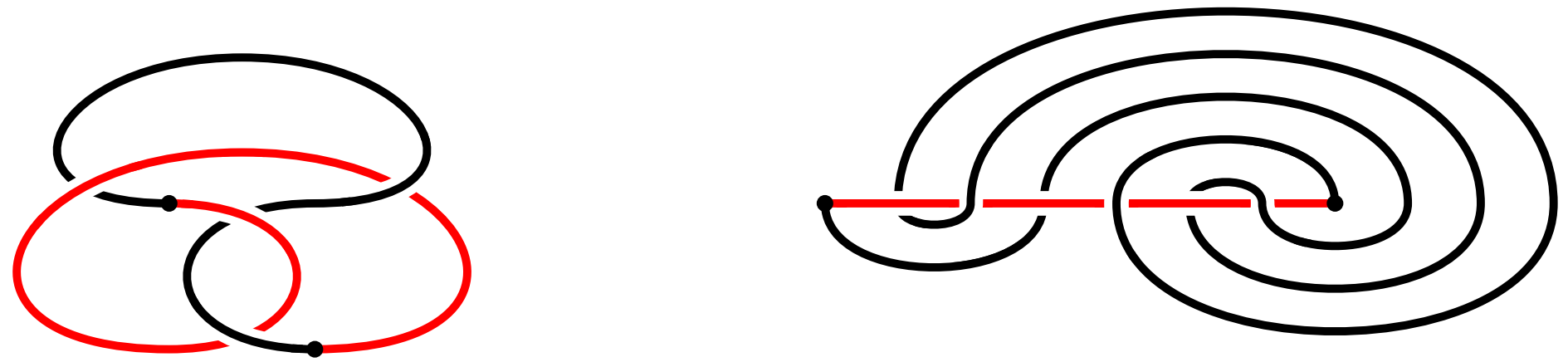}
    \caption{Examples of semimeander diagrams.}
    \label{fig:semimeander}
\end{figure}

In the present paper we study the following question: for a fixed knot $K$ how much does a semimeander diagram differ from the minimal diagram of $K$?  To be precise, let $K$ be a knot, $\cros(K)$ be its crossing number, and let $\cros_2(K)$ be the minimal number of crossings among all semimeander diagrams of $K$\footnote{We use the subscript <<2>> to indicate that we are considering diagrams composed of two smooth simple arcs. Similarly, one can consider $\cros_k(K)$ ---  the minimal number of crossings among all diagrams of $K$ composed of at most $k$ smooth simple arcs (these knot invariants were studied in~\cite{B20}). In this context, we can think of $\cros(K)$ as $\cros_\infty(K)$.}. We are interested in finding estimates on $\cros_2(K)$ in terms of $\cros(K)$. 
The first estimates of this kind were given in the works~\cite{BM17pca} and~\cite{O18}, where it was proved that ${\cros_2(K)\leq \sqrt{3}^{\cros(K)}}$. 
Later, the estimate was improved in~\cite{B20}, where it was claimed that ${\cros_2(K)\leq \sqrt[4]{6}^{\cros(K)}}$. However, the proof contained an error, and in fact it was proved that $\cros_2(K)\leq \sqrt[4]{8}^{\cros(K)}$. 
In this paper we improve these estimates and prove the following theorem:
\begin{thm}
\label{thm:main}
    Let $K$ be a knot with $\cros(K) > 10$, where $\cros(K)$ is the crossing number of $K$, and let $\cros_2(K)$ be the minimal number of crossings among all semimeander diagrams of $K$. Then 
    $$
    \cros_2(K) \leq \estim \approx 0.31 \cdot 1.557^{\cros(K)} .
    $$
\end{thm}
For all knots $K$ with $\cros(K) \leq 10$  the value of  $\cros_2(K)$ were found explicitly in~\cite{O18}.

Theorem~\ref{thm:main} can be used to obtain estimates of the complexity of other classes of meander-like diagrams. 
For instance, a knot diagram is called \emph{meander diagram} if it is composed of two smooth simple arcs whose common endpoints lie on the boundary of the convex hull of the diagram (see examples on Fig.~\ref{fig:meander}). 

\begin{figure}[h]
    \centering
    \includegraphics[scale = 0.19]{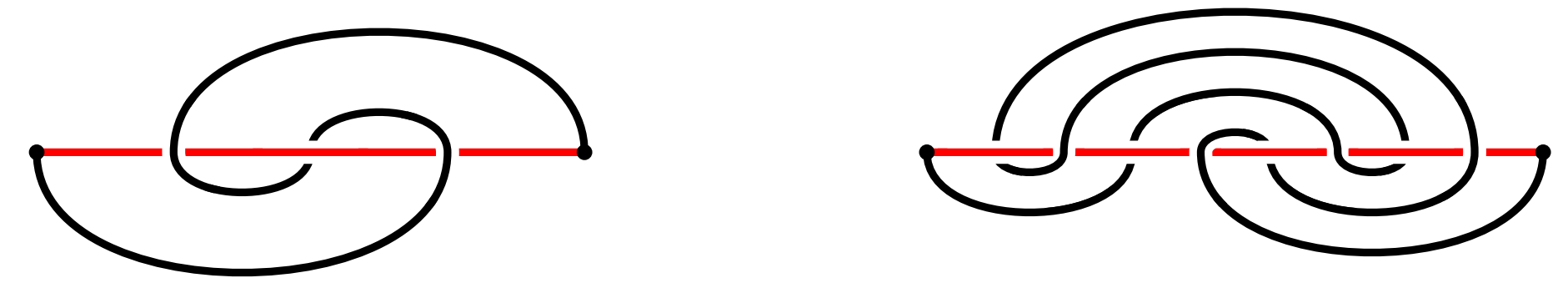}
    \caption{Examples of meander diagrams.}
    \label{fig:meander}
\end{figure}

\begin{corollary}
       Let $K$ be a knot with $\cros(K) > 10$, where $\cros(K)$ is the crossing number of $K$, and let $\crosmnd(K)$ be the minimal number of crossings among all meander diagrams of $K$. Then 
    $$
    \crosmnd(K) \leq \frac{832}{431}\sqrt[9]{\frac{8}{431}}\cdot\left(\sqrt[9]{53\frac{7}{8}}\right)^{\cros(K)} \approx 1.24 \cdot 1.557^{\cros(K)}.
    $$
\end{corollary}
\begin{proof}
    Let $D$ be a semimeander diagram of $K$ with precisely $\cros_2(K)$ crossings. Then there exists a meander diagram of the same knot with at most $4\cros_2(K)$ crossings (\cite[Lemma~1]{BKMMF22}). 
\end{proof}

There is an interesting subclass of meander diagrams, called potholder diagrams, see examples of such diagrams on Fig.~\ref{fig:potholder}. The fact that each knot has a potholder diagram was proved in~\cite{EHLN19}. 

\begin{figure}[h]
    \includegraphics[scale = 0.18]{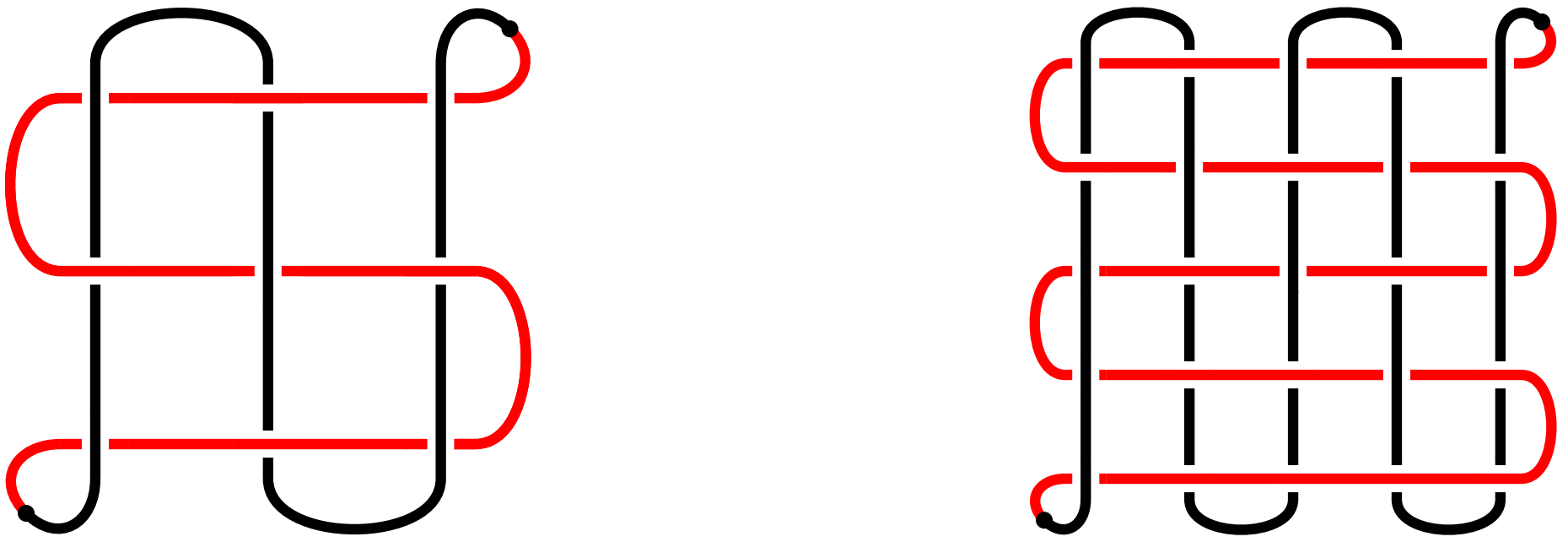}
    \caption{Examples of potholder diagrams.}
    \label{fig:potholder}
\end{figure}

\begin{corollary}
       Let $K$ be a knot with $\cros(K) > 10$, where $\cros(K)$ is the crossing number of $K$, and let $\crospot(K)$ be the minimal number of crossings among all potholder diagrams of $K$. Then 
    $$
    \crospot(K) \leq \left(\frac{1664}{431}\sqrt[9]{\frac{8}{431}}\cdot\left(\sqrt[9]{53\frac{7}{8}}\right)^{\cros(K)}-1\right)^2 \approx \left(2.48\cdot 1.557^{\cros(K)}-1\right)^2.
    $$
\end{corollary}
\begin{proof}
    Let $D$ be a meander diagram of $K$ with precisely $\crosmnd(K)$ crossings. Then there exists a potholder diagram of the same knot with at most $(2\crosmnd(K)-1)^2$ crossings (\cite[Proof of Theorem~1.1]{EHLN19}).
\end{proof}

In Section~\ref{sec:proof} the proof of Theorem~\ref{thm:main} is presented. The proof is constructive, it involves an analysis of a specific algorithm for constructing a semimeander diagram from a minimal one. In Section~\ref{sec:algorithm}, a more efficient algorithm for this task is presented.

\subsection*{Acknowledgment}
The author expresses gratitude to M.~Prasolov for discovering an inaccuracy in author's work~\cite{B20}, the correction of which led to the proof of the main result of this paper.

\section{Proof of Theorem~\ref{thm:main}} \label{sec:proof}

Let $K$ be a knot, let $D$ be its diagram with precisely $n$ crossings, and suppose there is a smooth simple arc~$J$ in $D$ such that no endpoint of $J$ is a crossing of $D$. Let $m$ be the number of crossings lying on $J$. We assume $m < n$, otherwise $D$ is already semimeander. 
The main idea is to successively transform $D$ in such a way that the number of crossings that do not lie on $J$ is reduced (the total number of crossings may increase during these transformations). 
\subsection*{Description of transformations}
We will consider two types of such transformations: Transformation~I ("pulling" the crossing along a simple arc, see Fig.~\ref{fig:method 1}) --- this is a sequence of so-called $\X$-moves (see Fig.~\ref{fig:x move}), and Transformation~{II} ("unroll" a simple loop, see Fig.~\ref{fig:method 2}). 
We will consider different sequences of such transformations, and choose the one that increases the total number of crossings the least. 
\begin{figure}[h]
    \centering
    \includegraphics[scale = 0.2]{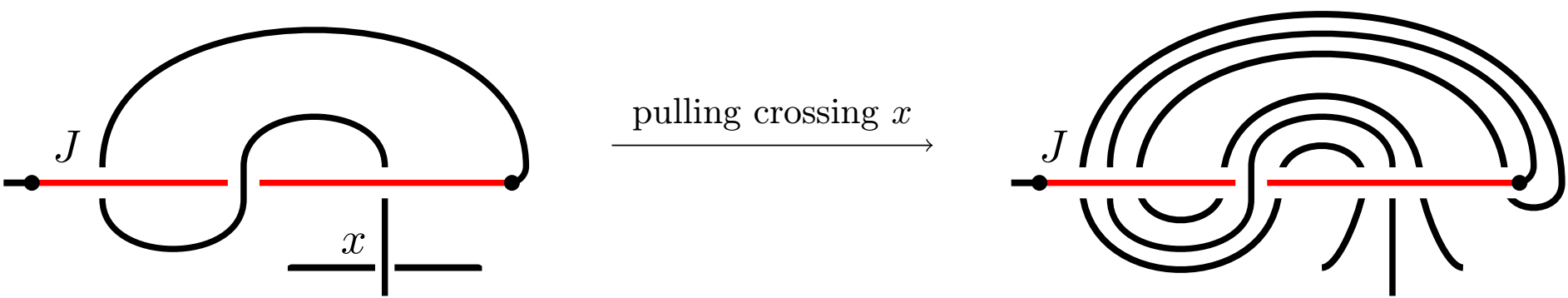}
    \caption{An example of applying Transformation~I.}
    \label{fig:method 1}
\end{figure}

Formally, Transformation~I can be described as follows. Let $s_1$ and $s_2$ be the two endpoints of $J$. We can choose an orientation on $D$ in such a way that $J$ is oriented from $s_1$ to $s_2$. Let's start moving from $s_2$, according to the chosen orientation, to the first crossing $x$ that does not lie on $J$. Since the arc connecting $s_2$ and $x$ is simple and contains only crossings that lie on $J$, we can make an $\X$-move (Fig.~\ref{fig:x move}) with each of them, thus moving $x$ to $J$.

\begin{figure}[h]
    \centering
    \includegraphics[scale = 0.17]{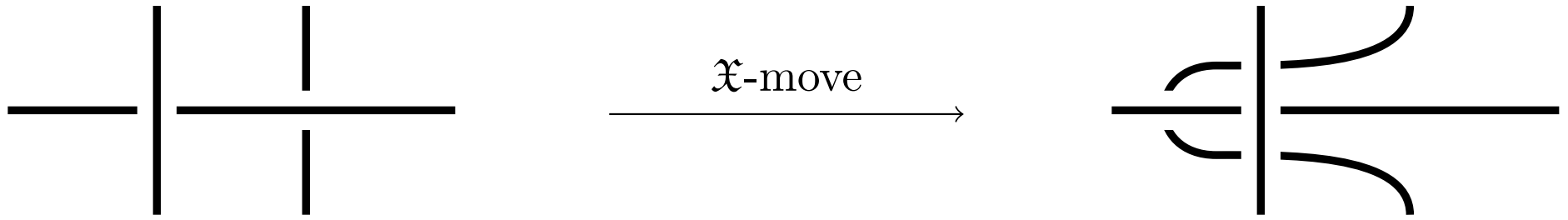}
    \caption{An example of applying $\X$-move.}
    \label{fig:x move}
\end{figure}


Note that if there were precisely $r$ crossings between $x$ and $s_2$, the total number of crossings in the diagram will increase by $2r$ after applying Transformation~I.

\begin{figure}[h]
    \centering
    \includegraphics[scale = 0.17]{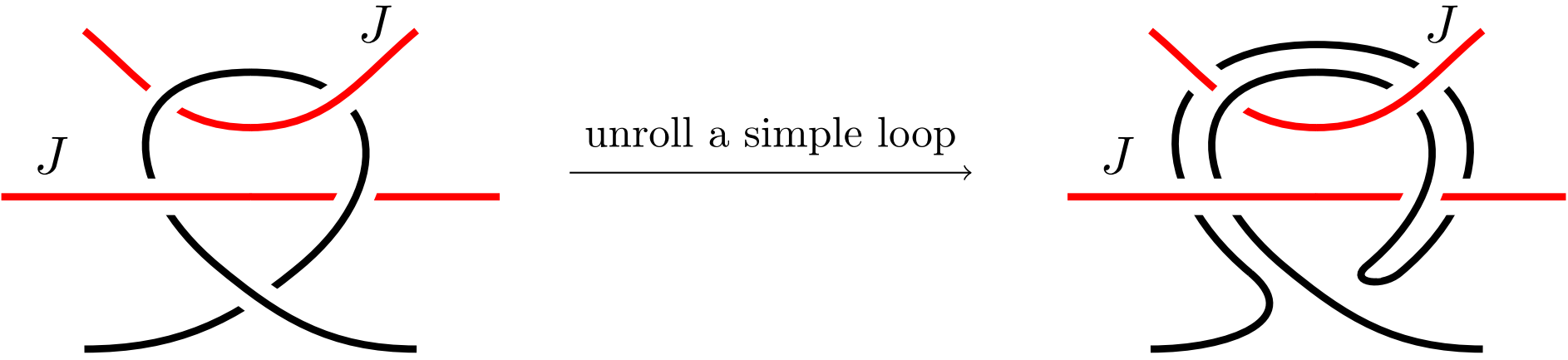}
    \caption{An example of applying Transformation~II.}
    \label{fig:method 2}
\end{figure}

Formally, Transformation~{II} can be described as follows.  Let $A$ be a simple loop\footnote{An arc in a knot diagram is called a \emph{simple loop} if its interior forms a simple curve and both endpoints of the arc coincide.} in $D$ which satisfies the following conditions: (1) the interior of $A$ contains no crossings lying on $D\setminus J$, (2) $J$ is not a subset of $A$, and (3) the crossing~$x$ corresponding to the endpoints of $A$ does not lie on $J$. 
Let $B_\varepsilon(A)$ be such a small open neighbourhood of $A$ that the only crossings of $D$ that are contained in $B_\varepsilon(A)$ lie on $A$. Now, let $S$ be an arc in $D$ such that: (1) $S$ contains the only crossing $x$, (2) $S$ lies in $B_\varepsilon(A)$. We can replace $S$ with an arc $S'$ with the following properties: (1) $S'$ has the same endpoints as $S$, (2) $S'$ is contained in $B_\varepsilon(A)$, and (3) the interior of $S'$ does not intersect $D\setminus J$ (over/undercrossings in the new crossings should be set in the obvious way).

  
Note that if there were precisely $r$ crossings lying on $A$, the total number of crossings in the diagram increases by $r-1$ after applying Transformation~II.

\subsection*{Definitions of ACD and preACD}
It will be convenient for us to use chord diagrams to study the described Transformations. Let us introduce the necessary technical definitions.
\begin{definition}
    Let $C$ be a chord diagram\footnote{The definitions related to chord diagrams are standard, and can be found, for example, in~\cite{CDM12}.} with a selected point (we call this point the \emph{basepoint}). The basepoint and the ends of the chords divide the circle of $C$ into connected components. If each of these connected components and the basepoint are assigned with a positive real number (called \emph{weight}), then $C$ is called \emph{annotated chord diagram}, or \emph{ACD} for short. The \emph{length} of ACD is the number of chords in it. The \emph{complexity} of ACD is the weight of the basepoint.
\end{definition}

If we label the chords in ACD, we can represent it using Gauss code, additionally placing the weights in square brackets and marking the weight of the basepoint with an over-line (see an example on Fig.~\ref{fig:ACD example}). The symbols corresponding to the labels of the chords are called \emph{non-special}.

\begin{definition}
    A non-cyclic substring of the Gauss code of an ACD is called \emph{preACD} if it contains the basepoint and the same number of non-special symbols before and after the basepoint. The \emph{length} of preACD is the number of non-special symbols in it divided by two. The \emph{complexity} of preACD is the weight of the basepoint.

\end{definition}

Let $\mathcal{C}_{k,m}$ be the set of all ACD of length $k$ and of complexity $m$, and let $\mathcal{D}_{k,m}$ be the set of all preACD  of length $k$ and of complexity $m$.

To each knot diagram with a selected simple arc one can associate an ACD as follows. 
Let $D$ be a knot diagram with a selected smooth simple arc~$J$ (no endpoint of $J$ are crossings). Let $C$ be a chord diagram corresponding to $D$\footnote{Technically, we should consider an oriented knot diagram here, but the choice of orientation does not affect all further constructions (up to trivial symmetries), so we omit it.}. Remove from $C$ all chords that have an endpoint on $J$, and contract the arc in $C$ that corresponds to $J$ to a point (this point will be the basepoint in corresponding ACD, and its weight will be equal to the number of crossings lying on $J$). This point and the ends of the remaining chords divide the circle of $C$ into connected components. To each component we assign a number that is equal to the number of chord ends that were removed from that component. See example on Fig.~\ref{fig:ACD example}. 
Notice that an ACD $X$ of length zero corresponds to semimeander diagram with number of crossings equals to the complexity of $X$.

\begin{figure}[h]
    \centering
    \includegraphics[scale = 0.18]{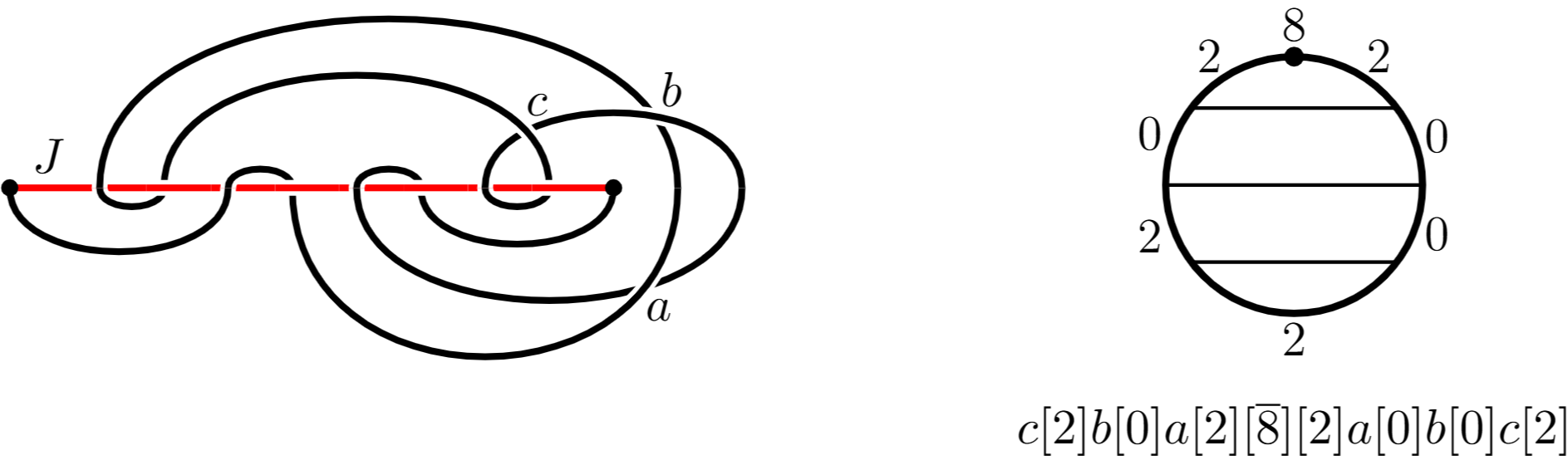}
    \caption{An example of the ACD constructed from a knot diagram.}
    \label{fig:ACD example}
\end{figure}

\subsection*{Transformations {I} and {II} in terms of ACD}
Transformation~{I} for ACD can be defined as follows. Let $X$ be an ACD of length greater than zero, and let 
$$x_{2n}[w_{2n+1}][\overline{m}][w_1]x_1[w_2]x_2\dots x_{2n-1}[w_{2n}]$$ be its Gauss code. Transformation~I can be performed to eliminate either $x_1$ or $x_{2n}$. Let us consider the first case, the second one is similar. Let $x_1 = x_i$ for some~$i$ ($2\leq i \leq 2n$). Then after performing Transformation~I we obtain an ACD with Gauss code 
$$x_{2n}[w_{2n+1}][\overline{m + 2w_1 + 1}][w_1+w_2]x_2\dots x_{i-1}[w_i+2w_1 + 1 + w_{i+1}]x_{i+1}\dots x_{2n-1}[w_{2n}].$$ 

Transformation~{II} can be described in similar way, but before doing this we need to make two observations. 

\subsubsection*{Observation 1} Let $X$ be an ACD associated to a knot diagram with a selected smooth simple curve, and suppose the Gauss code of $X$ contains a subword of the form
$$
\dots x_{1}[w_{1}]x_{2}[w_{2}]x_2[w_{3}]x_1 \dots\ .
$$
In this case, Transformation~{II} can be applied twice in such a way that the total number of crossings on a knot diagram does not increase by more than ${w_1 + 2w_2 +  w_3 - 2}$ (see Fig.~\ref{fig:observation 1}\footnote{Fig.~\ref{fig:observation 1} shows one of the two possibilities. The other option is when a simple loop is "inside" a larger loop. But in that case the complexity increases even less.}).

\begin{figure}[h]
    \centering
    \includegraphics[scale = 0.2]{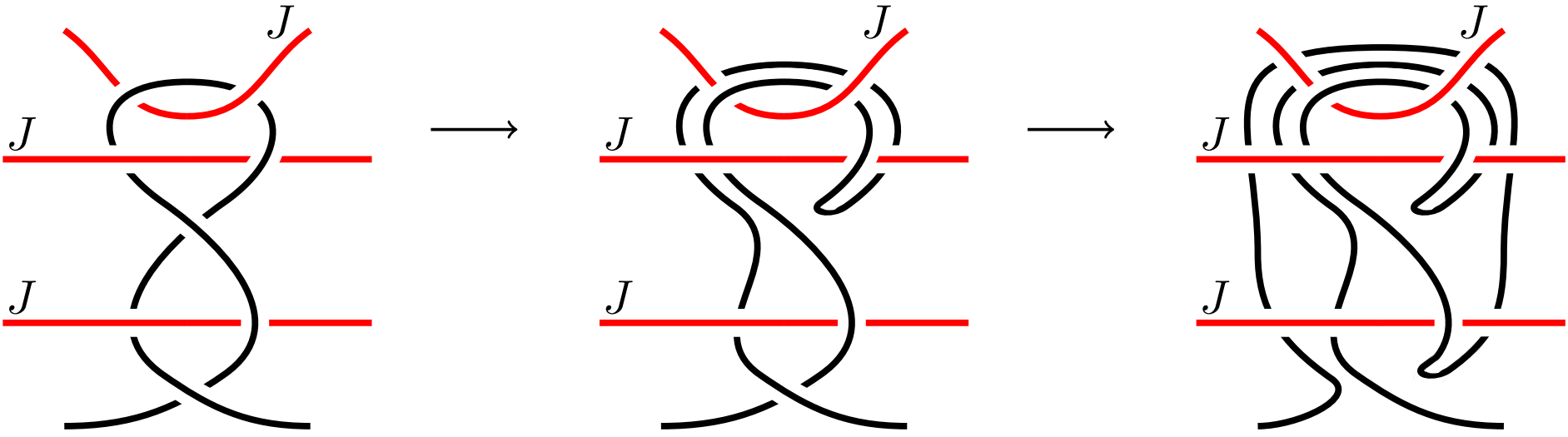}
    \caption{An example of sequential application of Transformation~{II}.}
    \label{fig:observation 1}
\end{figure}

\subsubsection*{Observation 2} Let $X$ be an ACD associated to a knot diagram with a selected smooth simple curve, and suppose the Gauss code of $X$ contains a subword of the form
$$
\dots[\overline{m}][w_{1}]x_{1}[w_{2}]x_{1}[w_{3}]x_2 \dots\ .
$$
In this case, Transformation~{I} can be applied after Transformation~{II} in such a way that the total number of crossings on a knot diagram does not increase by more than  $2(w_1 + 2w_2 +  w_3) - 1$ (see Fig.~\ref{fig:observation 2}).

\begin{figure}[h]
    \centering
     \includegraphics[scale = 0.2]{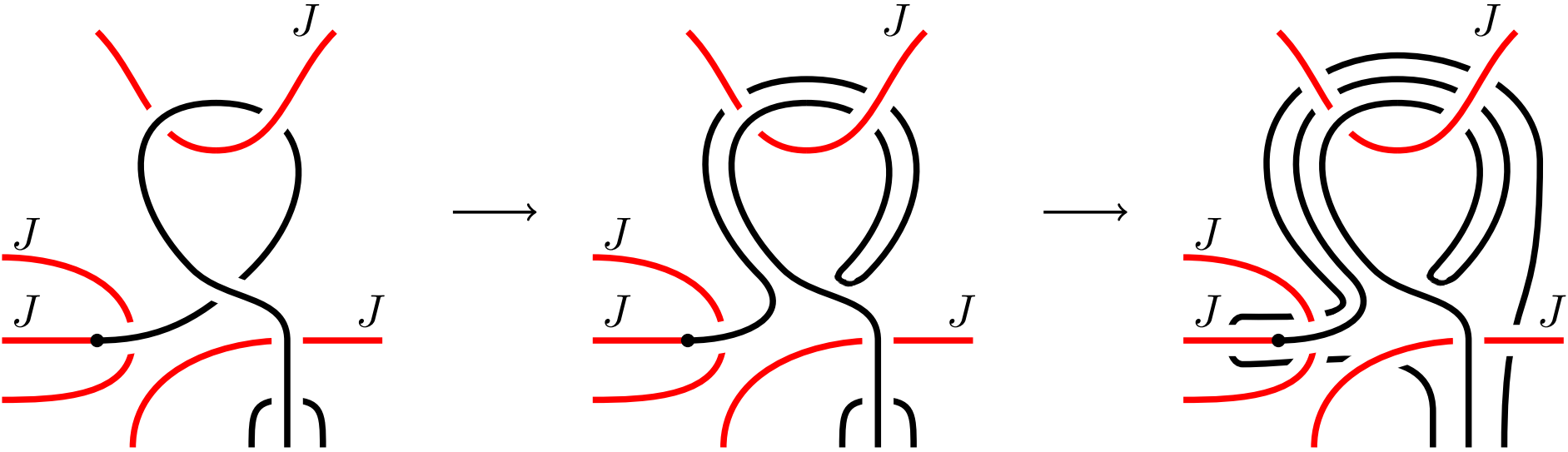}
    \caption{An example of an application of Transformation~{I} after Transformation~{II}.}
    \label{fig:observation 2}
\end{figure}

Based on these two observations, Transformation~{II} for ACD can be defined as follows. Let $X$ be an ACD of length greater than zero, and let 
$$x_{2n}[w_{2n+1}][\overline{m}][w_1]x_1[w_2]x_2\dots x_{2n-1}[w_{2n}]$$ 
be its Gauss code. Suppose $x_{i-1} = x_{i}$ for some $i$.  After performing Transformation~{II} at $x_i$ we obtain an ACD with Gauss code 
$$x_{2n}[w_{2n+1}][\overline{m + w_i}][w_1]x_1\dots x_{i-2}[w_{i-1} + w_{i} + w_{i+1}]x_{i+1}\dots x_{2n-1}[w_{2n}].$$ 

If $X$ is an ACD that was constructed from a knot diagram~$D$, each Transformation applied to $X$ corresponds to a analogous Transformation applied to $D$. It is therefore convenient to follow the changes in the number of crossings in the process of constructing  a semimeander diagram in terms of ACD. Transformations I and II are transferred verbatim to preACD as well.

\subsection*{Reduce the task to a linear programming problem.}
Let $D$ be a diagram of a knot $K$ with a selected smooth simple arc, and let $X \in \mathcal{C}_{k, m}$ be the corresponding ACD with weights $w_1,w_2,\dots,w_{2k+1}$. Consider all possible sequences of Transformations~I and {II} eliminating all chords from $X$ (we assume that all of such sequences are numbered from 1 to $r$). A sequence with number $i$ (for $1\leq i\leq r$) gives an ACD of length zero and of complexity $f_i(w_1,\dots, w_{2k+1})$, where 
$$f_i(w_1,\dots, w_{2k+1}) = \sum_{j=1}^{2k+1} a_{i,j}w_j + a_{i, 0}$$ is some linear function. Thus starting from $D$ and applying Transformations~I and II we can obtain a semimeander diagram of a knot $K$ with no more than $\min\limits_{i=1,\dots,r}f_i(w_1,\dots, w_{2k+1})$ crossings (see example below). 

\begin{example}
    Let $K$ be a knot, let $D$ be its diagram with a selected smooth simple arc on which all but one crossing lie, and let $X$ be the corresponding ACD. $X$ has Gauss code $[\overline{m}][w_1]1[w_2]1[w_3]$ (here we label the only chords with a number~1). There are three possible ways to eliminate the chord (see Fig.~\ref{fig:example of eliminations}). So we can obtain a semimeander diagram of $K$ with no more than $$\min \{m + 2w_1 + 1, m + 2w_3 + 1, m + w_2\} = m + 2 \frac{m-1}{4} + 1 = \frac{3m+1}{2}$$
    crossings (here we used the fact that $w_1+w_2+w_3 = m$). 
     \begin{figure}[h!]
        \centering        
        \includegraphics[scale  = 0.2]{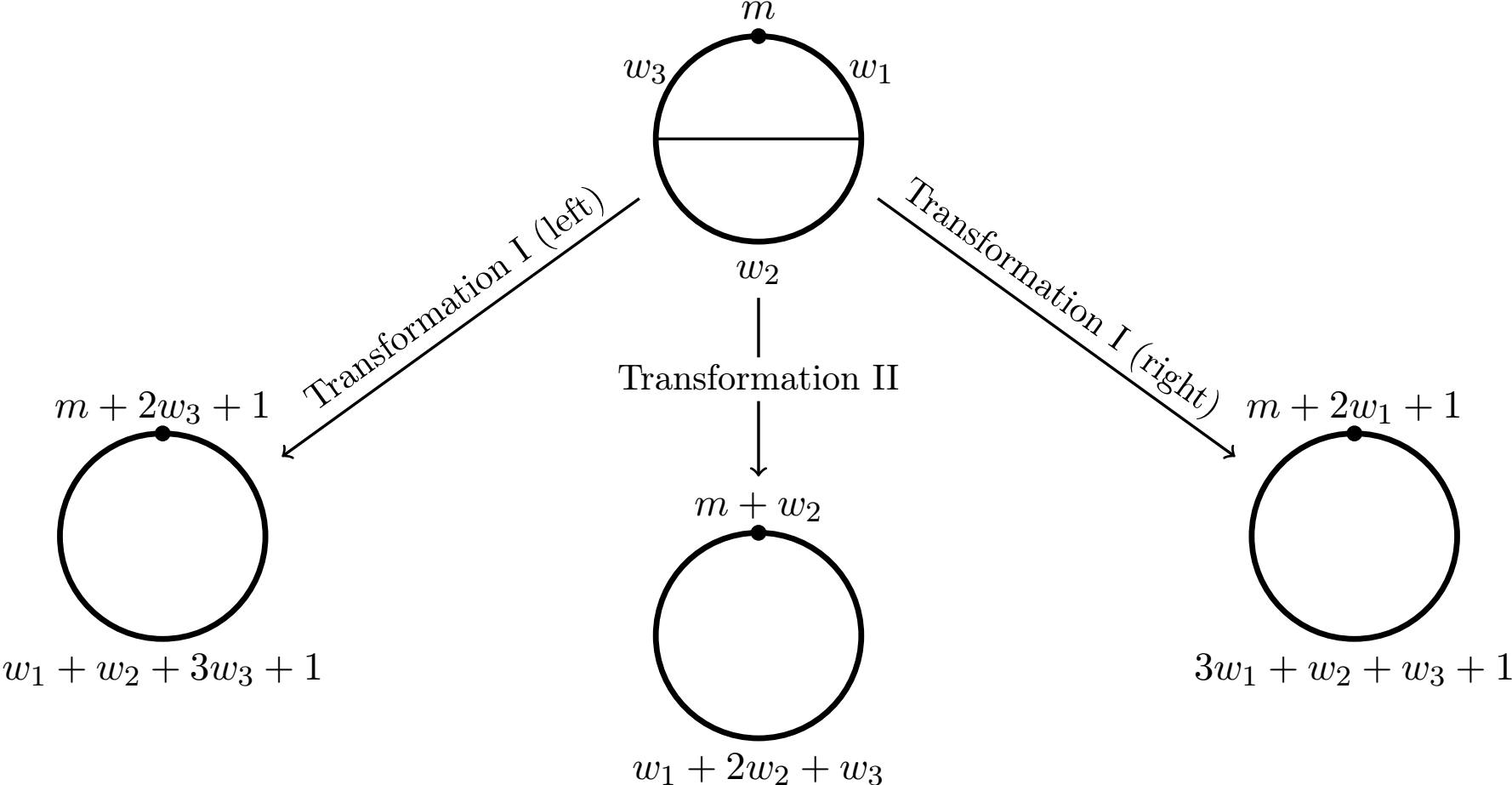}
        \caption{All possible eliminations of a single chord}
        \label{fig:example of eliminations}
    \end{figure}
\end{example}

In general case, if $X$ is an ACD of length $k$ we have the following maximization problem:\\
{Maximize
\begin{align*}
&\min_{i = 1,2,\dots, r} \left( \sum_{j=1}^{2k+1} a_{i,j}w_j + a_{i, 0} \right)
\end{align*}
Subject to
\begin{align*}
    &\sum_{j=1}^{2k+1} w_j \leq m, && \\
    &0 \leq w_j \leq m, && 1\leq j\leq 2k+1.     
\end{align*}}

This is a standard linear programming problem. Let us rewrite it in the following form:\\
Maximize
\begin{align} \label{eq:linear problem}
&t
\end{align}
Subject to
\begin{align}
    &\sum_{j=1}^{2k+1} w'_j \leq 1, && \\
    &0 \leq w'_j \leq 1, && 1\leq j\leq 2k+1, \\    
    &t \leq \sum_{j=1}^{2k+1} a_{i,j}w'_j + \frac{a_{i, 0}}{m}, &&  1\leq i\leq r.  \label{eq:linear constrains}
\end{align}

Let $\hat{t}(X)$ be the solution of equation~\eqref{eq:linear problem} for $X$, and let $C_{k, m} = \max\limits_{X\in \mathcal{C}_{k,m}}\hat{t}(X)$. 
Note that $C_{k, m_1} \geq C_{k, m_2}$ if $m_1\leq m_2$, and there exists a limit ${C_{k, \infty} = \lim\limits_{m\to \infty}C_{k, m}}$. 


Now let $K$ be a knot with $\cros(K) = n$, and let $D$ be its minimal diagram with a smooth simple arc passing through $m$ crossings. Then there exist a semimeander diagram of $K$ with no more than $mC_{n-m, m}$ crossings. From~\cite[Theorem~1]{BM17imm} it follows that if $\cros(K) > 10$ than there exists a minimal diagram of $K$ with a smooth simple arc passing through $m \geq 8$ crossings. Thus for an arbitrary knot $K$ with $n> 10$ crossings we get:
$$
\cros_2(K) \leq 8C_{n-8, 8}.
$$

\subsection*{The estimation of $C_{k,m}$ using preACD}
If $C_{n-8, 8}$ is unknown, we can estimate it using preACD. For each preACD $X$ of length $k$, we can consider all possible sequences of Transformation~I and II that eliminate $k$ chords, and thus obtain a similar linear problem as~\eqref{eq:linear problem}. 
Let $X \in \mathcal{D}_{k,m}$, let $\hat{t}(X)$ be a solution of the corresponding linear problem, and let $D_{k, m} =  \max\limits_{X\in \mathcal{D}_{k,m}}\hat{t}(X)$. 

Now let $K$ be a knot with $\cros(K) = n > 10$, let $D$ be its minimal diagram with a smooth simple arc passing through $8$ crossings, and let $d = n - 8 \ \mathrm{mod}\ k$. Than
$$
\cros_2(K) \leq 8\left(D_{k, 8}\right)^{\left\lfloor \frac{n-8}{k} \right\rfloor} C_{d, 8}.
$$

\begin{table}[h]
    \centering
\begin{tabular}{|c|cccccccccc|}\hline
$k$            & 0 & 1                  & 2                 & 3                 & 4              & 5  & 6  & 7  & 8  & 9  \\ \hline
& & & & & & & & & & \\[-1.5 ex]
$C_{k, 8}$      & 1 &  $1\frac{9}{16}$  & $3\frac{1}{4}$    & $4\frac{3}{8}$   & $6\frac{2}{5}$ & $11\frac{1}{8}$   &  $14\frac{1}{2}$  & $20\frac{1}{8}$   & $34\frac{3}{4}$   &  $44\frac{7}{8}$  \\[1. ex]
$C_{k, \infty}$ & 1 & $1\frac{1}{2}$               & 3                 & 4                 & $5\frac{4}{5}$ & 10 & 13 & 18 & 31 & 40 \\[1. ex]
$D_{k, 8}$      & 1 & $2\frac{1}{8}$     & $3\frac{1}{4}$    &  $4\frac{15}{16}$ & $7\frac{3}{4}$ &  $11\frac{1}{8}$  &  $16\frac{3}{4}$  & $24\frac{5}{8}$   &  37  & $53\frac{7}{8}$   \\[1. ex]
$D_{k, \infty}$ & 1 & 2                  & 3                 & $4\frac{1}{2}$    & 7              & 10 & 15 & 22 & 33 & 48\\[1. ex]
\hline
\end{tabular}
    \caption{Values of $C_{k, 8}$, $C_{k, \infty}$, $D_{k, 8}$ and $D_{k, \infty}$ for small $k$.}
    \label{tab:constants}
\end{table}

We found $C_{k, 8}$, $C_{k, \infty}$, $D_{k, 8}$ and $D_{k, \infty}$ for $k =1,\dots, 9$ (see Table~\ref{tab:constants}). The calculation was done with a C++ program that uses the ALGLIB library~\cite{alglib} to solve the linear programming problems (the code of the program is available at~\cite{Bcode}). Thus we get the final estimate:
\begin{align*}
\cros_2(K) &\leq 8\left(D_{9, 8}\right)^{\left\lfloor \frac{\cros(K)-8}{9} \right\rfloor} C_{d, 8} \leq \sqrt[9]{D_{9, 8}}^{\cros(K)} \frac{8C_{d, 8}}{\sqrt[9]{D_{9, 8}}^{d+8}} \leq \\ 
&\leq \estim \approx 0.31 \cdot 1.557^{\cros(K)}.
\end{align*}

 \begin{remark}
    $\log(D_{k, 8})$, $\log(D_{k,\infty})$, $\log(C_{k,8})$ and $\log(C_{k,\infty})$ grow almost perfectly linearly with increasing $k$ (see Fig.~\ref{fig:growth rate}). Therefore, we believe that even if values of $C_{k,m}$ will be explicitly determined for all $m$ and $k$, this would not result in an improved estimate beyond $\mathcal{O}\left(1.49^{\cros(K)}\right)$.
     \begin{figure}[h]
    \centering
    \includegraphics[scale = 0.33]{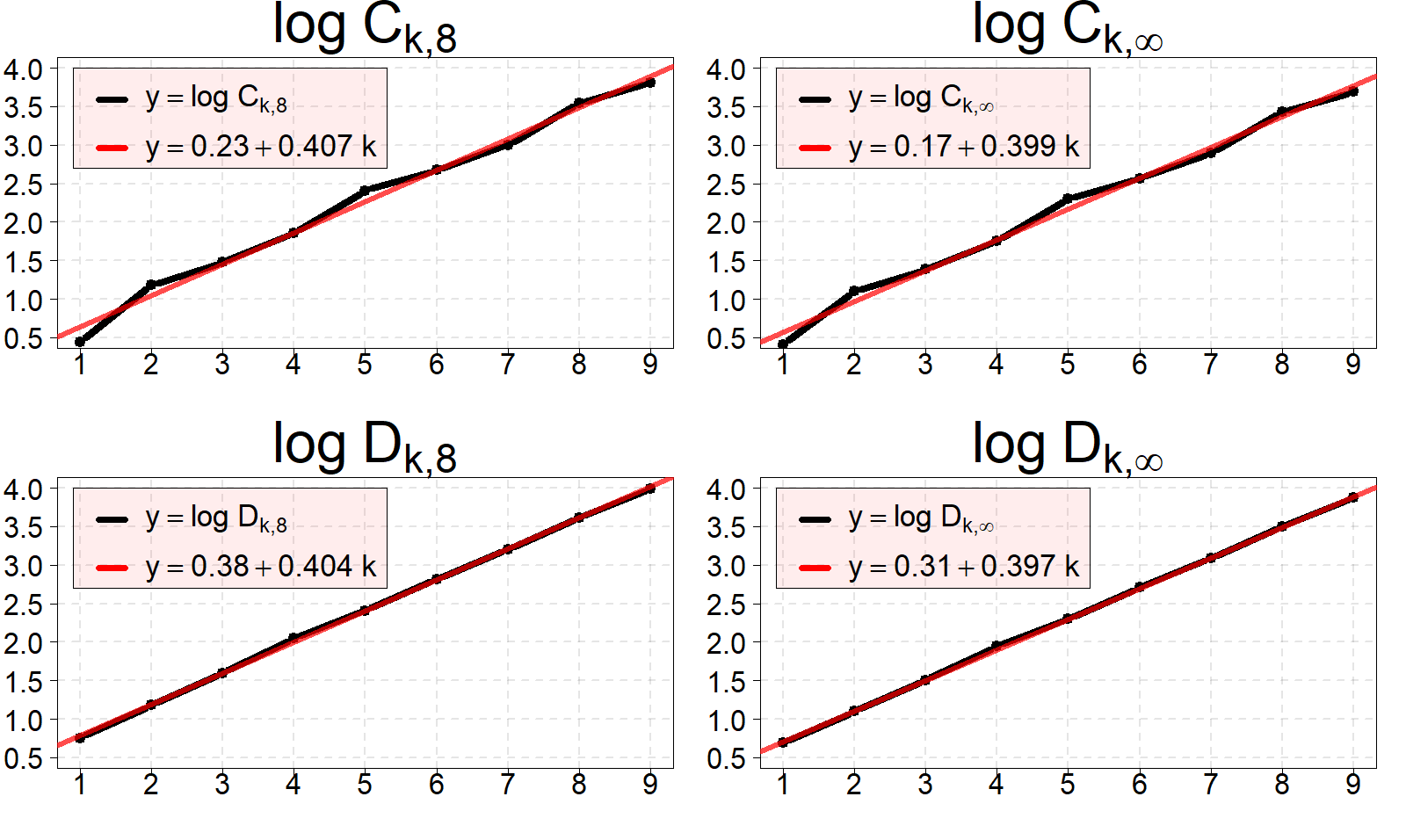}
    \caption{Growth rate of $D_{k, m}$ and $C_{k, m}$.}
    \label{fig:growth rate}
\end{figure}
\end{remark}

\section{An efficient algorithm for constructing a semimeander diagram}\label{sec:algorithm}
It's not hard to generalize Transformations I and II. In this section we will explore one approach to achieve this, resulting in an efficient algorithm for constructing a semimeander diagram from a given one.

Let $D$ be a diagram of a knot $K$, and let $J$ be a smooth simple arc in $D$ such that no endpoint of $J$ is a crossing of $D$. 
For each crossing~$x$ in $D$ we can choose two small simple arcs $A_1(x)$ and $A_2(x)$ with the following properties: (1) both $A_1(x)$ and $A_2(x)$ contain no crossings other than $x$, (2) their interiors intersect at $x$, (3) $A_1(x)$ corresponds to the undercrossing part of $x$. We denote the endpoint of $A_1(x)$ by $p_1(x)$ and $q_1(x)$, and the endpoints of $A_2(x)$ by $p_2(x)$ and $q_2(x)$ (see Fig.~\ref{fig:crossing}).  
A crossing~$x$ in $D$ is said to be \emph{reducible} if there exists a simple curve $\gamma$ with endpoints at $p_i(x)$ and $q_i(x)$ for some $i\in\{1,2\}$ such that (1) the interior of $\gamma$ intersects $D$ transversally in finite number of double points, and (2) all intersection point between the interior of $\gamma$ and $D$ lie on $J$. Such $\gamma$ is called \emph{reduction curve}. For each reducible crossing~$x$ we define \emph{reduction cost} to be $\min_\gamma |J\cap \gamma|$, where the minimum is taken among all reduction curves~$\gamma$.
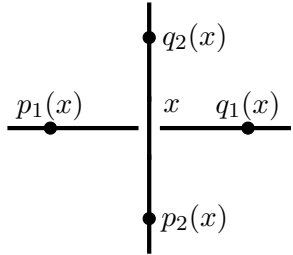
\begin{figure}[h] 
\begin{tikzpicture}
\node (h1) at (-2,  0) {};
\node (h2) at ( 2,  0) {};
\node (v1) at ( 0, -1.8) {};
\node (v2) at ( 0,  1.8) {};

\node (a1) at (-1.3,  0) {};
\node (b1) at ( 1.3,  0) {};
\node (a2) at ( 0, -1.2) {};
\node (b2) at ( 0,  1.2) {};

\node (na1) at (-1.3,  0.3) {$p_1(x)$};
\node (nb1) at ( 1.3,  0.3) {$q_1(x)$};
\node (na2) at ( 0.6, -1.2) {$p_2(x)$};
\node (nb2) at ( 0.6,  1.2) {$q_2(x)$};

\begin{knot}[consider self intersections, end tolerance=1pt, clip width = 5pt]
\strand [ultra thick] (v1) to (v2);
\strand [ultra thick] (h1) to (h2);
\end{knot}
\draw[fill] (a1) circle (0.08);
\draw[fill] (a2) circle (0.08);
\draw[fill] (b1) circle (0.08);
\draw[fill] (b2) circle (0.08);

\node (nx)	at (0.3, 0.3) {$x$};
\end{tikzpicture}
\caption{A neighborhood of a crossing.}
\label{fig:crossing}
\end{figure}

Now, for a given diagram~$D$ of a knot~$K$ we can obtain a semimeander diagram using the following algorithm. 
\begin{enumerate}
\item Choose a simple arc $J$ in $D$ such that no endpoint of $J$ is a crossing of $D$.
\item If $J$ contains all crossings of $D$ then $D$ is semimeander. Otherwise, we find a reducible crossing~$x$ with the smallest reduction cost (denote this reduction cost by $c$). 
\item Choose any reduction curve~$\gamma$ for $x$ such that the number of intersection points between $\gamma$ and $J$ is equal to $c$.
\item Let the endpoints of $\gamma$ be $p_i(x)$ and $q_i(x)$ for some $i\in\{1,2\}$. Replace the arc $A_i(x)$ with $\gamma$ (over/undercrossings in the new crossings should be set in the obvious way). 
\item Repeat steps 2--4 until the diagram becomes a semimeander one.
\end{enumerate}

In practice, this algorithm can be rewritten in terms of the dual graph for a knot diagram.  This makes it simple to find reduction costs for all crossings and also to find corresponding reduction curves (for example, using  Dijkstra's algorithm).


\bibliography{biblio}
\bibliographystyle{alpha}
\end{document}